\documentclass[twoside,11pt,reqno]{amsart}
\usepackage{amsmath,amssymb,amscd,mathrsfs,epic,wasysym,latexsym,tikz,mathrsfs,cite,hyperref,enumerate}
\usepackage{pb-diagram}
\usepackage[all]{xy}
\usepackage{mathdots}

\makeatletter

\numberwithin{equation}{section}

\allowdisplaybreaks

\newtheorem{Lemma}[equation]{Lemma}
\newtheorem{Theorem}[equation]{Theorem}

\newtheorem{MainTheorem}{Theorem}

\theoremstyle{definition}  

\newtheorem{Remark}[equation]{Remark}

\let\<\langle
\let\>\rangle

\newcommand\Comment[2][\relax]{\space\par\medskip\noindent%
   \fbox{\begin{minipage}{\textwidth}\textbf{Comment\ifx\relax#1\else---#1\fi}\newline%
        #2\end{minipage}}\medskip
}

\newcommand{\da}{{\downarrow}}

\def\pmod#1{\text{ }(\text{\rm mod } #1)\,}

\def\sgn{\mathtt{sgn}}

\newcommand{\res}{\operatorname{res}}

\newcommand{\Z}{\mathbb{Z}}

\def\eps{{\varepsilon}}
\def\phi{{\varphi}}

\newcommand{\F}{{\mathbb F}}

\newcommand{\la}{\lambda}

\newcommand{\de}{\delta}

\newcommand{\Mull}{{\tt M}}

\newcommand{\SSS}{{\sf S}}

\def\Parp{{\mathscr P}_p}
\def\Par{{\mathscr P}}

\def\k{\Bbbk}

\def\col{{\tt col}}
\def\row{{\tt row}}

{\catcode`\|=\active
  \gdef\set#1{\mathinner{\lbrace\,{\mathcode`\|"8000%
  \let|\midvert #1}\,\rbrace}}
}
\def\midvert{\egroup\mid\bgroup}

\colorlet{darkgreen}{green!50!black}
\tikzset{dots/.style={very thick,loosely dotted},
         greendot/.style={fill,circle,color=darkgreen,inner sep=1.5pt,outer sep=0},
         blackdot/.style={fill,circle,color=black,inner sep=1.5pt,outer sep=0},
         graydot/.style={fill,circle,color=gray,inner sep=1.1pt,outer sep=0}
}
\def\greendot(#1,#2){\node[greendot] at(#1,#2){}}
\def\blackdot(#1,#2){\node[blackdot] at(#1,#2){}}
\def\graydot(#1,#2){\node[graydot] at(#1,#2){}}

\newenvironment{braid}{
  \begin{tikzpicture}[baseline=6mm,black,line width=1pt, scale=0.32,
                      draw/.append style={rounded corners},
                      every node/.append style={font=\fontsize{5}{5}\selectfont}]%
  }{\end{tikzpicture}
}

\def\Grid(#1,#2){
  \draw[very thin,gray,step=2mm] (0,0)grid(#1,#2);
  \draw[very thin,darkgreen,step=10mm] (0,0)grid(#1,#2);
}

\newcommand\Tableau[2][\relax]{
  \begin{tikzpicture}[scale=0.5,draw/.append style={thick,black}]
    \ifx\relax#1\relax%
    \else 
      \foreach\box in {#1} { \filldraw[blue!30]\box+(-.5,-.5)rectangle++(.5,.5); }
    \fi
    \newcount\row\newcount\col
    \row=0
    \foreach \Row in {#2} {
       \col=1
       \foreach\k in \Row {
          \draw(\the\col,\the\row)+(-.5,-.5)rectangle++(.5,.5);
          \draw(\the\col,\the\row)node{\k};
          \global\advance\col by 1
       }
       \global\advance\row by -1
    }
  \end{tikzpicture}
}

\newcommand\YoungDiagram[2][\relax]{
  \begin{tikzpicture}[scale=0.5,draw/.append style={thick,black}]
    \ifx\relax#1\relax%
    \else 
    \foreach\box in {#1} {
      \filldraw[blue!30]\box rectangle ++(1,1);
    }
    \fi
    \newcount\row
    \row=0
    \foreach \col in {#2} {
       \draw(1,\the\row)grid ++(\col,1);
       \global\advance\row by -1
    }
  \end{tikzpicture}
}

\newdimen\hoogte    \hoogte=12pt    
\newdimen\breedte   \breedte=14pt  
\newdimen\dikte     \dikte=0.5pt 

\newenvironment{Young}{\begingroup
       \def\vr{\vrule height0.89\hoogte width\dikte depth 0.2\hoogte}
       \def\fbox##1{\vbox{\offinterlineskip
                    \hrule height\dikte
                    \hbox to \breedte{\vr\hfill##1\hfill\vr}
                    \hrule height\dikte}}
       \vbox\bgroup \offinterlineskip \tabskip=-\dikte \lineskip=-\dikte
            \halign\bgroup &\fbox{##\unskip}\unskip  \crcr }
       {\egroup\egroup\endgroup}
\def\Youngdiagram#1{\relax\ifmmode\vcenter{\,\begin{Young}#1\end{Young}\,}\else%
              $\vcenter{\,\begin{Young}#1\end{Young}\,}$\fi}

\begin{document}

\title[Dimensions of irreducible representations of symmetric groups]{{\bf Lower bounds for dimensions of irreducible representations of symmetric groups}}

\author{\sc Alexander Kleshchev}
\address{Department of Mathematics\\ University of Oregon\\Eugene\\ OR 97403, USA}
\email{klesh@uoregon.edu}

\author{\sc Lucia Morotti}
\address
{Institut f\"{u}r Algebra, Zahlentheorie und Diskrete Mathematik\\ Leibniz Universit\"{a}t Hannover\\ 30167 Hannover\\ Germany} 
\email{morotti@math.uni-hannover.de}

\author{\sc Pham Huu Tiep}
\address
{Department of Mathematics\\ Rutgers University\\ Piscataway\\ NJ~08854, USA} 
\email{tiep@math.rutgers.edu}


\thanks{The first author was supported by the NSF grant DMS-1700905 and the DFG Mercator program through the University of Stuttgart. The second author was supported by the DFG grant MO 3377/1-1 and the DFG Mercator program through the University of Stuttgart. The third author was supported by the NSF (grants DMS-1839351 and DMS-1840702), and the Joshua Barlaz Chair in Mathematics.
This work was also supported by the NSF grant DMS-1440140 and Simons Foundation while all three authors were in residence at the MSRI during the Spring 2018 semester.}

\thanks{The authors are grateful to the referee for careful reading and helpful comments on the paper.}

\begin{abstract}
We give new, explicit and asymptotically sharp, lower bounds for dimensions of irreducible modular representations of finite symmetric groups.  
\end{abstract}

\maketitle

\section{Introduction}

Let $\F$ be a field of characteristic $p> 0$. 
We denote by $\Par(n)$ the set of all partitions of $n$ and by $\Par_p(n)$ the set of all $p$-regular partitions of $n$, see \cite{JamesBook}. Given a partition $\mu=(\mu_1,\mu_2,\dots)\in\Par(m)$ and $n\in\Z_{\geq m+ \mu_1}$, we denote 
$$
(n-m,\mu):=(n-m,\mu_1,\mu_2,\dots)\in\Par(n).
$$
Let $\SSS_n$ be the symmetric group on $n$ letters, and denote by $D^\la$ the irreducible $\F\SSS_n$-module corresponding to a $p$-regular partition $\la$ of $n$, see \cite{JamesBook}. In \cite{JamesDim}, James gave sharp lower bounds for $\dim D^{(n-m,\mu)}$ for $m\leq 4$, and here we obtain asymptotically sharp lower bounds for all $m$.  

Set
$$
\de_p:=
\left\{
\begin{array}{ll}
0 &\hbox{if $p\neq 2$,}\\
1 &\hbox{if $p=2$.}
\end{array}
\right.
$$
For integers $m\geq 0$ and $n$ we define the rational numbers
\begin{align*}
C_{m}^p(n)&:=
p^m\binom{n/p-\de_p}{m}
\\
&
=\frac{1}{m!}\prod_{i=0}^{m-1}(n-(\de_p+i)p)
\\
&
=
\left\{
\begin{array}{ll}
\frac{n(n-p)(n-2p)\cdots(n-(m-1)p)}{m!} &\hbox{if $p>2$,}\\
\frac{(n-p)(n-2p)\cdots(n-mp)}{m!}  &\hbox{if $p=2$.}
\end{array}
\right.
\end{align*}
Our first main result develops \cite{JamesDim} as follows:

\begin{MainTheorem}\label{TA}
Let $m\geq 4$, $p$ a prime, $n\geq p(\de_p+m-2)$, and let $\mu\in\Par_p(m)$. Then for 
$\la:=(n-m,\mu)\in\Par_p(n)$ we have  
$$\dim D^\la\geq C_{m}^p(n).$$ 
\end{MainTheorem}

Note that $C^p_m(n) \approx n^m/m!$ when $p,m$ are fixed and $n \to \infty$. Hence, 
in view of \cite[Theorem 1]{JamesDim}, the lower bound of  Theorem~\ref{TA} is asymptotically sharp. Theorem~\ref{TA} will be crucially used in \cite{KMTTwo}. 

While Theorem~\ref{TA} requires that $n$ is relatively large compared to $m$, we also prove the following universal lower bound which improves \cite[Theorem 5.1]{GLT}.

\begin{MainTheorem}\label{TB}
Let $p\geq 3$ and $\la\in\Par_p(n)$. 
Let $\la^\Mull=(\la^\Mull_1,\la^\Mull_2,\dots)$ be the $p$-regular partition determined from $D^\la\otimes \sgn\cong D^{\la^\Mull}$. Let 
$a$ be minimal such that $D^\la\da_{\SSS_{n-a}}$ contains a submodule of dimension $1$,  and let
$$k:=\max\{\la_1,\la^\Mull_1\},~~t:= \max\{n-k,a\}.$$ 
Then 
$$\dim D^\la\geq 2\cdot 3^{(t-2)/3}.$$
\end{MainTheorem}

For $p=2$ we have the following result, which is a special case of Lemma~\ref{L1}:

\begin{MainTheorem}\label{TC}
Let $p=2$ and $\la\in\Par_2(n)$. Then $\dim D^\la\geq 2^{n-\la_1}$.
\end{MainTheorem}

\section{Main results}

\subsection{Preliminaries on modular branching rules}
In this subsection, we review modular branching rules for symmetric groups, which will be used below without further comment. The reader is referred to \cite{KBook,KBrII,KDec} for more details. 

We identify $\la\in\Par(n)$ and its Young diagram, which consists of nodes, i.e. elements of $\Z_{>0}\times\Z_{>0}$. Given any node $A=(r,s)$, its {\em residue} $\res A:=s-r\pmod{p}\in \Z/p\Z.$
For $i \in \Z/p\Z$ a node $A 
\in \la$ (resp. $B\not\in\la$) is called {\em $i$-removable} (resp. {\em $i$-addable}) for $\la$ 
if $\res A=i$ and $\la_A:=\la\setminus\{A\}$ (resp. $\la^B:=\la\cup\{B\}$) is a Young diagram of a partition.

Let $\la\in\Parp(n)$. Labeling the $i$-addable
nodes of $\la$ by $+$ and the $i$-removable nodes of $\la$ by $-$, the {\em $i$-signature} of 
$\la$ is the sequence of pluses and minuses obtained by going along the 
rim of the Young diagram from bottom left to top right and reading off
all the signs.
The {\em reduced $i$-signature} of $\la$ is obtained 
from the $i$-signature
by successively erasing all neighbouring 
pairs of the form $-+$. 
The nodes corresponding to  $-$'s in the reduced $i$-signature are
called {\em $i$-normal} for $\la$.
The leftmost $i$-normal node is called {\em $i$-good} . 
A node is called {\em removable} (resp. {\em normal}, {\em good}) if it is $i$-removable (resp. $i$-normal, $i$-good) for some $i$.
We denote 
$$\eps_i(\la):=\sharp\{\text{$i$-normal nodes of $\la$}\}.$$ 
If $\eps_i(\la)>0$, let $A$ be the $i$-good node of $\la$ and set
$
\tilde e_i \la:=\la_A.
$
Let $e_i$ be the {\em $i$-restriction functor} so that $V\da_{\SSS_{n-1}}=\bigoplus_{i\in\Z/p\Z}e_i V$ for any $\F\SSS_n$-module $V$. 

\begin{Lemma}\label{Lemma39}
Let $\lambda\in\Parp(n)$ and $i\in \Z/p\Z$. Then:
\begin{enumerate}
\item[{\rm (i)}]  $e_i D^\lambda\not=0$ if and only if $\eps_i(\lambda)>0$, in which case $e_iD^\lambda$ is a self-dual indecomposable module with socle and head both isomorphic to $D^{\tilde e_i\la}$.  
\item[{\rm (ii)}] Let $A$ be a removable node of $\la$ such that $\la_A$ is $p$-regular. Then $D^{\la_A}$ is a composition factor of $e_i D^\la$ if and only if $A$ is $i$-normal, in which case $[e_i D^\la:D^{\la_A}]$ is one more than the number of $i$-normal nodes for $\la$ above $A$. 
\end{enumerate}
\end{Lemma}

It follows easily from Lemma~\ref{Lemma39} that $D^\la\da_{\SSS_{n-1}}$ is  irreducible  if and only if the top removable node of $\la$ is its only normal node, in which case $\la$ is called a {\em Jantzen-Seitz} (or {\em JS}) partition, cf.  \cite{JS,k2}. 

\subsection{Properties of $C^p_m(n)$}

\begin{Lemma} \label{LIneq} 
For any $q\in{\mathbb R}_{\geq 1}$, $k\in\Z_{\geq 0}$ and $a\in{\mathbb R}_{\geq k}$ we have
$$
\prod_{i=0}^{k}(a-i)\leq \left(a-k+\frac{k}{q}\right)\prod_{i=0}^{k-1}\left(a-i-\frac{1}{q}\right).
$$
\end{Lemma}
\begin{proof}
Induction on $k$. For inductive step, it suffices to check that 
$$
a-(k+1)\leq \left(a-k-\frac{1}{q}\right)\left(a-k-1+\frac{k+1}{q}\right)\left(a-k+\frac{k}{q}\right)^{-1},
$$
which is elementary. 
\end{proof}

\begin{Lemma} \label{LBinom} 
Let $m\geq 1$. Then:
\begin{enumerate}
\item[{\rm (i)}] $C_{m}^p(n)=C_{m}^p(n-p)+pC_{m-1}^p(n-p)$.
\item[{\rm (ii)}] If $n\geq p(\de_p+m-1)$ then $C_{m}^p(n)\leq C_{m}^p(n-1)+C_{m-1}^p(n-1)$.
\end{enumerate}
\end{Lemma}
\begin{proof}
(i) follows from 
\begin{align*}
\frac{C_{m}^p(n)}{p^m}=\binom{n/p-\de_p}{m}&=\binom{n/p-\de_p-1}{m}+\binom{n/p-\de_p-1}{m-1}
\\
&=\binom{(n-p)/p-\de_p}{m}+\binom{(n-p)/p-\de_p}{m-1}
\\
&=\frac{C_{m}^p(n-p)}{p^m}+\frac{C_{m-1}^p(n-p)}{p^{m-1}}.
\end{align*}

(ii) Note that 
\begin{align*}
&C_{m}^p(n-1)+C_{m-1}^p(n-1)
\\
=\,\,&\frac{1}{m!}\prod_{i=0}^{m-1}(n-1-(\de_p+i)p)+\frac{1}{(m-1)!}\prod_{i=0}^{m-2}(n-1-(\de_p+i)p)
\\
=\,\,&\frac{1}{m!}((n-1-(\de_p+m-1)p)+m)\prod_{i=0}^{m-2}(n-1-(\de_p+i)p).
\end{align*}
Multiplying by $m!$ and dividing by $p^m$, it suffices to prove that 
$$
\prod_{i=0}^{m-1}\left(\frac{n}{p}-\de_p-i\right)\leq
\left(\frac{n}{p}-\de_p-m+1+\frac{m-1}{p}\right)\prod_{i=0}^{m-2}\left(\frac{n}{p}-\de_p-i-\frac{1}{p}\right)
$$
This holds by Lemma~\ref{LIneq} with $a=\frac{n}{p}-\de_p$, $k=m-1$ and $q=p$. 
\end{proof}

\subsection{Proof of Theorem~\ref{TA}}

\begin{Lemma} \label{LJames} {\rm \cite{JamesDim}} 
Let $1\leq m\leq 4$, $\mu\in \Par_p(m)$, and $n$ be such that $(n-m,\mu)\in\Par_p(n)$. Then 
$$
\dim D^{(n-m,\mu)}\geq 
\left\{
\begin{array}{ll}
n-2 &\hbox{if $m=1$,}\\
(n^2-5n+2)/2 &\hbox{if $m=2$,}\\
(n^3 -9n^2 +14n)/6 &\hbox{if $m=3$.}
\\(n^4-14n^3+47n^2-34n)/24 &\hbox{if $m=4$.}
\end{array}
\right. 
$$
\end{Lemma}

\begin{Theorem} 
Let $m\geq 4$, $n\geq p(\de_p+m-2)$, $\mu\in \Par_p(m)$, and suppose that $\la:=(n-m,\mu)\in\Par_p(n)$. Then 
$\dim D^\la\geq C_{m}^p(n)$. 
\end{Theorem}
\begin{proof}
If $p(\de_p+m-2)\leq n\leq  p(\de_p+m-1)$, we have $C_{m}^p(n)\leq 0$ and there is nothing to prove. So we assume that $n> p(\de_p+m-1)$. 

Let $m=4$ and set $f(n):=(n^4-14n^3+47n^2-34n)/24$, see Lemma~\ref{LJames}. If $p\geq 3$ then $n> p(\de_p+m-1)\geq 9$ and $f(n)\geq C_m^p(n)$, and so we are done in this case. If $p=2$, then $n> p(\de_p+m-1)\geq 8$, while $f(n)\geq C_m^p(n)$ for $n>10$. For $n=9$ and $10$, the claimed dimension bound holds by inspection of \cite[Tables]{JamesBook}. 

So, in addition to $n> p(\de_p+m-1)$ we now assume that $m\geq 5$. We apply induction on $n$. 
Note that $n> p(\de_p+m-1)$ implies $n-2m>1$, unless $p=2$, in which case we have $n-2m\geq 1$.  Hence $\la_1-\la_2\geq 2$, unless $p=2$ and $\la=(m+1,m)$. In the exceptional case, $D^\la$ is the basic spin module of dimension $2^m$, and the bound boils down to $2^m\geq \frac{(2m-1)!!}{m!}$, which is easily checked. Thus we may assume that $\la_1-\la_2\geq 2$. Let $A=(1,\la_1)$ be the top removable node of $\la$. 

Suppose first that $\la$ is not JS. 
Then $A$ is not the only normal node of $\la$, so there exists a good node $B$ of $\la$ with $B\neq A$. Then $D^{\la_A}$ and $D^{\la_B}$ are composition factors of $D^\la\da_{\SSS_{n-1}}$. 
The inductive assumption applies to $D^{\la_A}$ to give $\dim D^{\la_A}\geq C^p_m(n-1)$. Since $m\geq 5$, the inductive assumption applies to $D^{\la_B}$ to give 
$\dim D^{\la_B}\geq C^p_{m-1}(n-1)$. Now the result follows from Lemma~\ref{LBinom}(ii). 

Next, let $\la$ be JS, and let $B$ be the second removable node from the top. Suppose first that $\la_1-\la_2>p$ and for $t=0,1,2, \dots,p$, set $A_t:=(1,\la_1+1-t)$. We denote 
$$
\la^{(t)}:=(\la_1-t,\la_2,\la_3,\dots)=(\dots(\la_{A_1})_{A_2}\dots)_{A_{t}}\qquad(1\leq t\leq p).
$$
As $\la$ is JS, we have 
$D^\la\da_{\SSS_{n-1}}\cong D^{\la^{(1)}}$. As $\la$ is JS, we have $\res B=\res A_0=\res A_p$. So successive application of the branching rules implies that $D^\la\da_{\SSS_{n-p+1}}$ contains composition factors $D^{\la^{(p-1)}}$ and $D^{(\la_B)^{(p-2)}}$, the second one with multiplicity at least $p-2$. Modular branching rules now imply that $[D^{\la^{(p-1)}}\da_{\SSS_{n-p}}:D^{(\la_B)^{(p-1)}}]=2$, and so we deduce that $D^\la\da_{\SSS_{n-p}}$ contains composition factors $D^{\la^{(p)}}$ and $D^{(\la_B)^{(p-1)}}$, the second one with multiplicity at least $p$. Now result follows from the inductive assumption and Lemma~\ref{LBinom}(i). 

Thus we may assume that $\la$ is JS, and $\la_1-\la_2\leq p$. If $p\geq3$, we deduce
$$
p\geq \la_1-\la_2\geq n-2m>p(\de_p+m-1)-2m=p(m-1)-2m=(p-2)m-p,
$$
implying $p=3$, $m=5$ and $n=13$, hence $\la=(8,5)$, which is not JS. 

Finally, let $p=2$. 
Then $\la_1-\la_2=2$ since $\la$ is JS. The assumption $n>p(\de_p+m-1)=2m$ now implies that $\la=(m+2,m)$ or $\la=(m+1,m-1,1)$. In the first case, $\la$ is a basic spin module of dimension $2^m$, and the required  bound boils down to $2^m\geq \frac{(2m)!!}{m!}$, which is actually an equality! In the second case we have $\la=(m+1,m-1,1)$. By the modular branching rules, $D^{(m,m-2,1)}$ appears in $D^\la\da_{\SSS_{n-2}}$ with multiplicity at least $2$, and the result follows from
$$
2C^2_{m-1}(n-2)=2\frac{(2m-3)!!}{(m-1)!}> \frac{(2m-1)!!}{m!}=C^2_m(n).
$$
The theorem is proved.
\end{proof}

\begin{Remark}
Some other lower bounds on the dimensions of irreducible modular representations of $\SSS_n$ were obtained in \cite{Mu}, based on an improved 
version \cite[Theorems (5.2), (5.6)]{Mu} of James' \cite[Lemma 4]{JamesDim}.
\end{Remark}

\subsection{Proof of Theorems~\ref{TB} and \ref{TC}}
\begin{Lemma}\label{L1}
Let $\la\in\Par_p(n)$. Then 
$$\dim D^\la\geq\prod_{i\geq p} \lceil{i/(p-1)}\rceil^{\la_i}.$$ 
In particular, 
$$\dim D^\la\geq 2^{n-\la_1-\ldots-\la_{p-1}}.$$
\end{Lemma}

\begin{proof}
Let $A_1,A_2,\ldots$ be the removable nodes counting from the top and let $A=A_j$ be minimal such that $\la_{A_j}$ is $p$-regular. If $A_j$ is on row $i$ then $(j-1)(p-1)<i\leq j(p-1)$ and nodes $A_1,\ldots,A_j$ are all normal of the same residue. So 
$$[D^\la\da_{\SSS_{n-1}}:D^{\la_A}]=j=\lceil i/(p-1)\rceil,$$ 
from which the lemma follows by induction.
\end{proof}

\begin{Lemma}
Let $a,b\geq 0$ with $a-b\geq p-1$. Then $\dim D^{(a,b)}\geq 2^b$.
\end{Lemma}

\begin{proof}
If $a-b>p-1$ then $D^{(a-1,b)}$ is a composition factor of $D^{(a,b)}\da_{\SSS_{a+b-1}}$, while if $a-b=p-1$ then $D^{(a,b-1)}$ is a composition factor with multiplicity $2$ of $D^{(a,b)}\da_{\SSS_{a+b-1}}$. The lemma then follows.

Alternatively, the lemma follows from Lemma \ref{L1} and \cite[Lemma 2.3]{bkz}.
\end{proof}

\begin{Lemma}
Let $\la\in\Par_p(n)$. If $\la_1\geq p-1$ and $((\la_1)^\Mull,(\la_2,\la_3,\ldots)^\Mull)\in\Par_p(n)$ then $\dim D^\la\geq 2^{n-\la_1}$.
\end{Lemma}

\begin{proof}
The lemma follows from Lemma \ref{L1} and \cite[Lemma 2.2]{bkz}.
\end{proof}

The following result improves \cite[Theorem 5.1]{GLT}.

\begin{Theorem}
Let $p\geq 3$ and $\la\in\Par_p(n)$. Further let $k:=\max\{\la_1,\la^\Mull_1\}$ and
$a \in \Z_{>0}$ be minimal such that $D^\la\da_{\SSS_{n-a}}$ contains a submodule of dimension $1$. Then 
$$\dim D^\la\geq 2\cdot 3^{(\max\{n-k,a\}-2)/3}.$$
\end{Theorem}

\begin{proof}
If $\la\in\{(n),(n)^\Mull\}$ then the statement clearly holds. So we will assume
that this is not the case.
If $\mu$ is obtained from $\la$ by removing a sequence of $b$ good nodes, then
$\mu^\Mull$ can also be obtained from $\la^\Mull$ by removing a sequence of $b$ good nodes.
In particular $\max\{\mu_1,\mu^\Mull_1\}\leq k$. Also if $D^\mu\da_{\SSS_{n-b-c}}$ contains a
submodule of dimension $1$ then $c\geq a-b$ by minimality of $a$. By induction
we can assume that $\dim D^\mu\geq 2\cdot 3^{(\max\{n-k,a\}-2-b)/3}$.

{\sf Case 1.} {\em $\la$ is not JS.} If $\eps_i(\la)\geq 2$ for some $i$ then $[D^\la\da_{\SSS_{n-1}}:D^{\tilde e_i \la}]\geq 2$ and $D^{\tilde e_i \la}\subseteq D^\la\da_{\SSS_{n-1}}$. Otherwise there exist $i\not=j$ with $\eps_i(\la),\eps_j(\la)=1$ and then $D^{\tilde e_i \la}\oplus D^{\tilde e_j \la}\subseteq D^\la\da_{\SSS_{n-1}}$. In either case
\[\dim D^\la\geq 4\cdot 3^{(\max\{n-k,a\}-3)/3}>2\cdot 3^{(\max\{n-k,a\}-2)/3}.\]

{\sf Case 2.} {\em $\la$ is JS.} Let $A$ be the top normal node of $\la$. Then $A$ is good in $\la$ and $D^\la\da_{\SSS_{n-1}}\cong D^{\la_A}$. From 
\cite[Lemma 3]{JamesDim} we have that $\la_A$ has at least $2$ normal nodes. If $\la_A$ has at least 3 normal nodes we can conclude similarly to the previous case that
\[\dim D^\la\geq 6\cdot 3^{(\max\{n-k,a\}-4)/3}>2\cdot 3^{(\max\{n-k,a\}-2)/3}.\]
So we may assume that $\la_A$ has exactly $2$ normal nodes. Further notice that $D^{(2)}$ and $D^{(1^2)}$ are both composition factors of $D^\la\da_{\SSS_2}$ since $\la\not\in\{(n),(n)^\Mull\}$. Since $p\geq 3$, it follows that 
$$D^\la\da_{\SSS_{n-2,2}}\cong (D^\mu\boxtimes D^{(2)})\oplus(D^\nu\boxtimes D^{(1^2)}),$$ where $\mu,\nu\in\Par_p(n-2)$ can each be obtained from $\la_A$ by removing a good node. In particular if $D^\pi\subseteq D^\la\da_{\SSS_{n-3}}$ then $\pi$ can be obtained from $\la$ by removing a sequence of $3$ good nodes. Also $\mu=\tilde e_i \la_A$ and $\nu=\tilde e_j\la_A$ with $i\not=j$.

If $\mu$ and $\nu$ are not both JS then similar to before
\[\dim D^\la\geq 6\cdot 3^{(\max\{n-k,a\}-5)/3}=2\cdot 3^{(\max\{n-k,a\}-2)/3}.\]

If $p\geq 5$ and $\mu$ and $\nu$ are both JS, then $D^\la\da_{\SSS_{n-3}}$ has only $2$ composition factors. From $$D^\la\da_{\SSS_{n-2,2}}\cong (D^\mu\boxtimes D^{(2)})\oplus(D^\nu\boxtimes D^{(1^2)})$$ it follows that either $D^\la\da_{\SSS_{n-3,3}}\cong (D^\pi\boxtimes D^{(2,1)})$ or $$D^\la\da_{\SSS_{n-3,3}}\cong (D^\psi\boxtimes D^{(3)})\oplus(D^\xi\boxtimes D^{(1^3)})$$ for certain partitions $\pi,\psi,\xi$. So from \cite[Corollary 3.9]{BK} with $k=3$ or from \cite[Corollary 4.3]{BK} we have that $n\leq 5$ or $p\mid n$ and $\la\in\{(n-1,1),(n-1,1)^\Mull\}$. The cases $n\leq 5$ can be checked separately. If $p\mid n$ and $\la\in\{(n-1,1),(n-1,1)^\Mull\}$ then $n-k=1$, $a=2$ and $\dim D^\la=n-2\geq 3>2$. 

So we can now assume that $p=3$. We will show that in this case $\mu$ and $\nu$ are not both JS, from which the lemma follows. From the previous part all normal node of $\la_A$ are good. So it is enough to show that that for a certain normal node $B$ of $\la_A$ we have that $(\la_A)_B$ is not JS.

{\sf Case 2.1.} {\em $\la_1\geq \la_2+3$.} 
If  $B:=(1,\la_1-1)$ then $B$ is normal in $\la_A$ and $(1,\la_1-2)$ and the second top removable node of $\la$ are normal in $(\la_A)_B$.

{\sf Case 2.2.} {\em $\la_1=\la_2+2$.} Then $\la$ is not JS.

{\sf Case 2.3.} {\em $\la_1=\la_2+1$.} Then $\la=(\la_1,\la_1-1,\la_3,\ldots)$ with $1\leq \la_3\leq\la_1-2$. If $B=(2,\la_1-1)$ then $B$ is normal in $\la_A$ and $(1,\la_1-1)$ and the third top removable node of $\la$ are normal in $(\la_A)_B$.

{\sf Case 2.4.} {\em $\la_1=\la_2$.} Then $\la=(\la_1^2,\la_3,\ldots)$ with $1\leq \la_3\leq \la_1-2$. If $B=(1,\la_1)$ then  $B$ is normal in $\la_A$ and $(2,\la_1-1)$ and the second top removable node of $\la$ are normal in $(\la_A)_B$.
\end{proof}

\end{document}